\documentclass{amsart}

\usepackage{amsmath, amssymb, hyperref, upgreek, xcolor, mathtools, tikz, float}
\usetikzlibrary{decorations.pathmorphing}

\newtheorem{theorem}{Theorem}[section]
\newtheorem{thmA}{Theorem}

\newtheorem{thmB}{Theorem}
 
\newtheorem{lemma}[theorem]{Lemma}
\newtheorem{proposition}[theorem]{Proposition}

\theoremstyle{definition}
\newtheorem{definition}[theorem]{Definition}
\newtheorem{example}[theorem]{Example}

\theoremstyle{remark}
\newtheorem{remark}[theorem]{Remark}

\numberwithin{equation}{section}

\usepackage[
backend=biber,
style=alphabetic,
sorting=ynt
]{biblatex}

\title{Typical conservative homeomorphisms have total metric mean dimension}
\author{Gabriel Lacerda and Sergio Romaña}
\date{} % Activate to display a given date or no date (if empty),
\address{Instituto de Matematica, Universidade Federal do Rio de Janeiro. C. P. 68.530, CEP 21.945-970, Rio de Janeiro, RJ.}
\email{lacerda@im.ufrj.br}
\email{sergiori@im.ufrj.br}

\thanks{The authors would like to thank Alexander Arbieto for the suggestion that resulted in this work. The first author is currently a PhD student founded by a CNPq grant. The second author is supported by FAPERJ with the grant Bolsa Jovem Cientista do Nosso Estado No. E-26/201.432/2022.}

\begin{document}
\maketitle

\begin{abstract}
    Given a compact smooth boundaryless manifold with dimension greater than one endowed with a locally positive non-atomic measure $\mu$, we prove that typical $\mu$-preserving homeomorphisms have upper metric mean dimension, with respect to the Riemannian distance, equal to the dimension of the manifold. Moreover, we prove that $\mu$ is a measure of maximal metric mean dimension, with respect to the variational principle established in \cite{velozo_velozo_2017}. 
\end{abstract}

%%%%%%%%%%%%%%%%%%%%%%%%%%%%%%%%%%%%%%%%%%%%%%%%%%%%%%%%%%%%%%%%%%%%%%%%%%%%%%%%%%%%

\section{Introduction}

In order to calculate the complexity of a continuous dynamical system on a compact metric space, denoted as $f: X \to X$, many authors referred to its \textit{topological entropy}. This quantity is an invariant by topological conjugation and may be infinite. Indeed, Yano proved in \cite{yano_1980} that, on a compact smooth manifold with dimension greater than one, typical homeomorphisms have infinite topological entropy.

Another topological invariant is the \textit{mean dimension}, proposed by Gromov \cite{gromov_1999} at the end of the last century. It counts the average number of parameters needed to describe the orbit of a point in $X$, and gives a numerical invariant for infinite dimensional dynamical systems with infinite topological entropy. In particular, if the phase space $X$ has finite dimension and $f$ is a homeomorphism, then the mean dimension of the system is zero.

However, the mean dimension is difficult to compute. Therefore, Lindenstrauss and Weiss introduced in \cite{lindenstrauss_weiss_2000} the concept of \textit{metric mean dimension}, which depends on the metric of the phase space; hence, it is not an invariant under topological conjugacy. It is also an upper bound for the mean dimension and is closely related to the topological entropy. Precisely, it was proven by A. Velozo and R. Velozo in \cite{velozo_velozo_2017} that the dimension of $X$ is an upper bound for the metric mean dimension.

Also in \cite{velozo_velozo_2017}, the authors established a new variational principle for the metric mean dimension, analogous to the one introduced in \cite{lindenstrauss_tsukamoto_2018}. This new principle simplifies computations and is utilized in one of our main results.

Recently, it has been shown in \cite{carvalho_rodrigues_varandas_2020} by M. Carvalho, F. B. Rodrigues and P. Varandas that typical homeomorphisms, on a compact smooth boundaryless manifold with dimension $n$ greater than one, has metric mean dimension equal to $n$. Furthermore, for all $\alpha \in [0, n]$, it was proved by J. Acevedo, S. Romaña and R. Arias in \cite{acevedo_romana_arias_2023} that the set of homeomorphisms on $X$ with metric mean dimension equal to $\alpha$ is dense in $\text{Homeo}(X)$.

%%%%%%%%%%%%%%%%%%%%%%%%%%%%%%%%%%%%%%%%%%%%%%%%%%%%%%%%%%%%%%%%%%%%%%%%%%%%%%%%%%%%

\subsection{Conservative homeomorphisms}

From now on, consider $X$ a compact connected smooth boundaryless manifold of dimension $n \geq 2$, endowed with a Riemannian distance $d$. The set of homeomorphisms of $X$, denoted by $\text{Homeo}(X)$, is a complete metric space with the metric 
\[
D(f, g) = \max\limits_{x \in X}\{d\big(f(x), g(x)\big), d\big(f^{-1}(x), g^{-1}(x)\big)\}.
\]

A \textit{good} Borel probability measure $\mu$ on $X$ is a measure satisfying:
\begin{enumerate}
    \item[i)] For all $x \in X$, $\mu(\{x\}) = 0$ (non-atomic);
    \item[ii)] For all $B \subset X$ non-empty open set, $\mu(B) > 0$ (locally positive).
\end{enumerate}

Once for all, fix such a measure $\mu$. If $f \in \text{Homeo}(X)$ preserves $\mu$, then call it a \textit{conservative} homeomorphism; otherwise call $f$ \textit{dissipative}. Denote by $\text{Homeo}(X, \mu)$ the set of conservative homeomorphisms, that is also a complete metric space with the metric D (see \cite{alpern_prasad_04} in Chapter 2). These measures are as well called Oxtoby-Ulam measures.

\begin{remark}
    The topology generated by the metric $D$ on both $\text{Homeo}(X)$ and $\text{Homeo}(X, \mu)$ is called the \textit{$C^0$-topology}.
\end{remark}

\begin{remark}
    If $X$ is a manifold with boundary, then the good measure also satisfies $\mu(\partial X) = 0$. This case is not relevant because our proof relies on local modifications of homeomorphisms in the interior of the manifold.
\end{remark}

%%%%%%%%%%%%%%%%%%%%%%%%%%%%%%%%%%%%%%%%%%%%%%%%%%%%%%%%%%%%%%%%%%%%%%%%%%%%%%%%%%%%

\subsection{Typical property}

Let $\mathcal{X}$ be a complete metric space. We call $G_\delta$ any countable intersection of open subsets of $\mathcal{X}$. According to Baire's theorem, a countable intersection of dense open sets is a dense $G_\delta$ set, which we refer to as a \textit{residual} set. A property is \textit{typical} or \textit{generic} in $\mathcal{X}$ if it is satisfied on at least a residual set. Note that for a finite (or countable) number of typical properties, the set possessing all of these properties remains a countable intersection of dense open sets, and therefore, it is dense. This says that the typical property behaves well under intersection, and as a consequence, allows to talk about generic conservative homeomorphisms and list their different properties. 

%%%%%%%%%%%%%%%%%%%%%%%%%%%%%%%%%%%%%%%%%%%%%%%%%%%%%%%%%%%%%%%%%%%%%%%%%

\subsection{Metric mean dimension}

Let $\mathfrak{X}$ be a compact metric space and denote its metric by $d$. Given $T: \mathfrak{X} \to \mathfrak{X}$ a continuous map and $k \in \mathbb{N}$, define the dynamical distance $d_k$ as 
\[
d_k(x, y) = \max \{d\big(T^i(x), T^i(y)\big), 0 \leq i \leq k - 1\},
\]
where $T^0$ is the identity. It is well know that $d_k$ still a distance function on $\mathfrak{X}$ and generates the same topology as $d$. For $x \in \mathfrak{X}$, we call $B_{(k, \varepsilon)}(x) = \{y \in \mathfrak{X}; d_k(x, y) < \varepsilon\}$ the \textit{$(k, \varepsilon)$-dynamical ball}.

Since $\mathfrak{X}$ is compact, the number $N(T, k, \varepsilon)$, defined as the minimal cardinality of a covering of $\mathfrak{X}$ by $(k, \varepsilon)$-dynamical balls, is finite.

Call $A \subset \mathfrak{X}$ a \textit{$(k, \varepsilon)$-separated set} if to any distinct points $x, y \in A$, $d_k(x, y) \geq \varepsilon$. Denote by $S(T, k, \varepsilon)$ the maximal cardinality of a $(k, \varepsilon)$-separated set, that is finite by the compactness of $\mathfrak{X}$. In order to clarify further notations, define 
\[
\text{Sep}(T, \varepsilon) = \lim\limits_{k \to \infty} \frac{\log S(T, k, \varepsilon)}{k}.
\]

This limit always exists because $\log S(T, k, \varepsilon)$ is a subadditive function of $k$. Note that if $\varepsilon_1 < \varepsilon_2$, then $\text{Sep}(T, \varepsilon_1) \geq \text{Sep}(T, \varepsilon_2)$. The topological entropy $h(T)$ is the limit of $\text{Sep}(T, \varepsilon)$ as $\varepsilon \to 0$. When the topological entropy is infinite, we are interested in the growth of $\text{Sep}(T, \varepsilon)$. This motivates the definition of upper and lower metric mean dimension.

The \textit{lower metric mean dimension} and the \textit{upper metric mean dimension} of $(\mathfrak{X}, d, T)$ are defined by
\[
\underline{\text{mdim}}_M(\mathfrak{X}, d, T) = \liminf\limits_{\varepsilon \to 0} \frac{\text{Sep}(T, \varepsilon)}{-\log\varepsilon}
\;\;\;\text{and}\;\;\;\overline{\text{mdim}}_M(\mathfrak{X}, d, T) = \limsup\limits_{\varepsilon \to 0} \frac{\text{Sep}(T, \varepsilon)}{-\log\varepsilon},
\]
respectively. By Remark 4 in \cite{velozo_velozo_2017}, when the metric space $\mathfrak{X}$ is a manifold of dimension $n$, then 
\[
\overline{\text{mdim}}_M(\mathfrak{X}, d, T) \leq n.
\]

This fact is crucial for the proof of Theorem \ref{conservative-homeo-full-metric-mean-dimension}.

%%%%%%%%%%%%%%%%%%%%%%%%%%%%%%%%%%%%%%%%%%%%%%%%%%%%%%%%%%%%%%%%%%%%%%%%%%%%%%%%%%%%

\subsection{Main results}

The origins of Ergodic Theory and Dynamical Systems lie in the study of physical systems that evolve over time as solutions to certain differential equations. Taking conservation laws into account, Liouville's Theorem ensures that for Hamiltonian systems, this flow has an invariant measure. Thus, one is naturally led from the underlying physics to the study of measure-preserving manifold homeomorphisms or diffeomorphisms.

Many researchers studied typical properties for conservative homeomorphisms. In \cite{oxtoby_ulam_1941}, Oxtoby and Ulam proved that ergodicity is generic for measure preserving homeomorphisms on a compact manifold. Katok and Stepin \cite{katok_stepin_1970} proved that weak mixing homeomorphisms are also generic. In the 1970's, more work was done by S. Alpern, V. S. Prasad, and P. Lax (see the historical preface by Alpern and Prasad in \cite{alpern_prasad_04}). 

Not long ago, P. Guihéneuf and T. Lefeuvre proved in \cite{guiheneuf_lefeuvre_2018} that the specification property is generic in $\text{Homeo}(X, \mu)$, which does not hold generically in the dissipative case. Moreover, Guihéneuf proved in \cite{guiheneuf_2012}, Chapter 3, that typical conservative homeomorphisms have infinite topological entropy.

However, based on the findings in \cite{carvalho_rodrigues_varandas_2020}, it appears that genericity with regard to the totality of the metric mean dimension is still lacking in the conservative case. This work aims to bridge this gap by presenting the following result:

\begin{thmA}\label{conservative-homeo-full-metric-mean-dimension}
For $X$ a compact connected smooth boundaryless manifold of dimension $n \geq 2$, endowed with a Riemannian distance $d$, there exists a residual subset $\mathfrak{R} \subset \text{Homeo}(X, \mu)$ such that
\[
\overline{\text{mdim}}_M(X, f, d) = n,\;\; \text{ for all } f \in \mathfrak{R}.
\]
\end{thmA}

Our proof differs from most proofs concerning genericity in the conservative case: it does not relies on the use of Oxtoby-Ulam-Brown's theorem (see \cite{brown_1961}) because the metric mean dimension depends on the Riemannian metric and the chart, guaranteed by the theorem, that cover the manifold is not necessarily bi-Lipschitz. But it is somewhat inspired by the construction of horseshoes in \cite{guiheneuf_2012} and in \cite{carvalho_rodrigues_varandas_2020}.

The following main result answers a problem proposed in \cite{velozo_velozo_2017,} regarding the existence of a measure of 'maximal metric mean dimension'. In fact, many approaches to construct a "measure theoretic mean dimension" fails, see \cite{lindenstrauss_tsukamoto_2018} for a detailed explanation. However, there are alternatives for the metric mean dimension, as the variational principle introduced in \cite{lindenstrauss_tsukamoto_2018} and \cite{velozo_velozo_2017}.

Precisely, let $\mathcal{M}_T(\mathfrak{X})$ be the space of $T$-invariant Borel probability measures on $\mathfrak{X}$. For $\delta \in (0,1)$, $k \in \mathbb{N}$ and $\varepsilon > 0$, define $N_\nu(T, k, \varepsilon, \delta)$ as the minimum number of $(k, \varepsilon)$-dynamical balls needed to cover a set of $\nu$-measure strictly bigger than $1 - \delta$, where $\nu \in \mathcal{M}_T(\mathfrak{X})$. Set
\[
h_\nu(\varepsilon, T, \delta) = \limsup\limits_{k \to \infty} \frac{\log N_\nu(T, k, \varepsilon, \delta)}{k}.
\]

It was proven by Katok in \cite{katok_1980} that the measure-theoretic entropy $h_\nu(T) = \lim_{\varepsilon \to 0} h_\nu(\varepsilon, T, \delta)$. In particular, this limit does not depend on $\delta \in (0, 1)$. The classical variational principle is $h(T) = \sup_{\nu \in \mathcal{M}_f(X)}h_\nu(T)$.

In \cite{velozo_velozo_2017}, the authors proved the following variational principle.

\begin{theorem}\label{variational-principle-velozo}
    Let $(\mathfrak{X}, d)$ be a compact metric space and $T: \mathfrak{X} \to \mathfrak{X}$ continuous. Then
    \[
    \overline{\text{mdim}}_M(\mathfrak{X}, T, d) = \limsup\limits_{\varepsilon \to 0} \frac{\sup_\delta\sup_{\nu \in \mathcal{M}_f(X)}h_\nu(\varepsilon, f, \delta)}{-\log \varepsilon}.
    \]
\end{theorem}

The above theorem motivates our second main result, which states that the measure $\mu$ is a measure of maximal metric mean dimension.

\begin{thmB}\label{measure-maximal-metric-mean-dimension}
    For all $f$ in $\mathfrak{R}$, $\mu$ is a measure of maximal metric mean dimension, that is, 
    \[
    \overline{\text{mdim}}_M(X, f, d) = \limsup\limits_{\varepsilon \to 0} \frac{\sup_\delta h_\mu(\varepsilon, f, \delta)}{-\log \varepsilon} = n.
    \]
\end{thmB}

\subsection{Reading guide}

The paper is organized as follows. In Section \ref{section-conservative-toolbox}, we recall basic results from the theory of conservative homeomorphisms and prove that the set of homeomorphisms in $\text{Homeo}(X, \mu)$ with a periodic point is dense. In Section \ref{section-pseudo-horseshoes}, we recall the definition of Markovian intersection, define a pseudo-horseshoe, and provide a sufficient condition under which the conservative homeomorphism $f$ exhibits sufficiently large separated sets. In Section \ref{section-perturbation-lemmas}, we present the proof of several local conservative modification lemmas, which are crucial for constructing pseudo-horseshoes along a periodic orbit. Finally, in Section \ref{section-proof-of-the-main-theorems}, we prove both Theorems \ref{conservative-homeo-full-metric-mean-dimension} and \ref{measure-maximal-metric-mean-dimension}. 

%%%%%%%%%%%%%%%%%%%%%%%%%%%%%%%%%%%%%%%%%%%%%%%%%%%%%%%%%%%%%%%%
\section{Conservative toolbox}\label{section-conservative-toolbox}

Endow $\mathbb{R}^n$ with the maximum norm, that is, for $u = (u_1, ..., u_n) \in \mathbb{R}^n$,
\[
\lVert u\rVert = \lVert (u_1, ..., u_n)\rVert = \max\limits_{1 \leq i \leq n} \lvert u_i\rvert.
\]
For $\delta > 0$, define $B^n_\delta = \{u \in \mathbb{R}^n, \lVert u \rVert \leq \delta\}$ and set $I^n := B^n_{\frac{1}{2}}$, a unitary cube. Also consider $\lambda$ the Lebesgue measure in $\mathbb{R}^n$. The following theorem, credited to Oxtoby and Ulam, can be found in \cite{alpern_prasad_04}, Chapter 9.

\begin{theorem}[Homeomorphic Measures Theorem]\label{homeomorphic-measure}
A Borel probability measure $\mu$ on the $n-$cube $I^n$ is homeomorphic to Lebesgue measure $\lambda$ on $I^n$ if and only if it is a good Borel probability measure. In other words, there is a homeomorphism $h: I^n \to I^n$ with $\mu = h_*\lambda$ if and only if the Borel probability measure $\mu$ is nonatomic, locally positive, and has $\mu(\partial I^n) = 0$. Furthermore, given any homeomorphism $g: I^n \to I^n$, we can choose $h$ to equal $g$ on $\partial I^n$; in particular, we can take $h$ to be the identity on the boundary.
\end{theorem}

The above theorem is fundamental so that we can perturb a conservative homeomorphism locally in a manifold via its charts, that is, we suppose that a chart takes an open set onto $I^n$. Now, the proof of the following theorem can be found in \cite{oxtoby_ulam_1941}.

\begin{lemma}[Extension of finite maps]\label{extension-finite-maps-lemma}
    Let $X$ be a manifold and $z_1, ..., z_m$ be $n$ different points of $X\setminus\partial X$ and $\Phi: \{z_1, ..., z_m\} \to X$ be an injective map such that $D(\Phi, x) < \delta$. Then there exists $\varphi \in \text{Homeo}(X, \mu)$ such that $\varphi(z_i) = \Phi(z_i)$, for all $1 \leq i \leq m$, and $D(\varphi, x) < \delta$. Moreover, given $m$ injective continuous paths $\gamma_i$ joining $z_i$ to $\Phi(z_i)$, the support of $\varphi$ can be chosen in any neighbourhood of the union of the paths $\gamma_i$.
\end{lemma}

Let $F \in \text{Homeo}(X, \mu)$ and $x \in X$. Given $B_\varepsilon(x) = \{y \in X; d(x, y) < \varepsilon\}$, a consequence of the Poincaré recurrence theorem is that $\mu$-almost every point in $B_\varepsilon(x)$ returns to $B_\varepsilon(x)$ by $F$. In particular, there is $y_0 \in B_\varepsilon(x)$ and $m \in \mathbb{N}$ such that $F^m(y_0) \in B_\varepsilon(x)$. Apply above lemma to construct $\varphi \in \text{Homeo}(X, \mu)$ such that $\varphi\circ F^m(y_0) = y_0$ and $D(\varphi, x) < 2\varepsilon$. Hence, $\varphi\circ F \in \text{Homeo}(X, \mu)$ has $y_0$ as a periodic point of period $m$. Thus, this also proves the density of homeomorphisms with a periodic point in the set of conservative homeomorphisms.

The above lemma will also be used to transform a set that is homeomorphic to a cube into another set that is equal to a cube in a finite, but sufficiently large, number of points.

\begin{definition}[Bicollared embeddings]
An embedding $\sigma_0$ of a manifold $X_0$ into a manifold $X$ is said to be \textit{bicollared} if there exists an embedding $\sigma: [-1, 1]\times X_0 \to X$ such that $\sigma_{\{0\}\times X_0} = \sigma_0$.
\end{definition}

The above definition is important because it avoids pathological embeddings such as the Alexander horned sphere, and is connected to the next lemma, that can be seen in \cite{guiheneuf_2012}, Chapter 3.

\begin{lemma}[Local modification]\label{local-modification-lemma}
Let $\sigma_1$, $\sigma_2$, $\tau_1$, $\tau_2$ be four bicollared embeddings of $\mathbb{S}^{n - 1}$ in $\mathbb{R}^n$, such that $\sigma_1$ is in the bounded connected component of $\sigma_2$ and $\tau_1$ is in the bounded connected component of $\tau_2$. Let $A_1$ be the bounded connected component of $\mathbb{R}\setminus\sigma_1$ and $B_1$ be the bounded connected component of $\mathbb{R}^n\setminus\tau_1$, $\Sigma$ the connected component of $\mathbb{R}^n\setminus(\sigma_1\cup\sigma_2)$ with boundaries $\sigma_1\cup\sigma_2$ and $\Lambda$ the connected component of $\mathbb{R}^n\setminus(\tau_1\cup\tau_2)$ with boundaries $\tau_1\cup\tau_2$, $A_2$ the unbounded connected component of $\mathbb{R}^n\setminus\sigma_2$ and $B_2$ the unbounded connected component of $\mathbb{R}^n\setminus\tau_2$.
Consider two homeomorphisms $h_i: A_i \to B_i$ such that they both preserve or reverse the orientation. Then, there exists a homeomorphism $h: \mathbb{R}^n \to \mathbb{R}^n$ such that $h \equiv h_1$ on $A_1$ and $h \equiv h_2$ on $A_2$.
Moreover, if we assume that $\lambda(A_1)$ = $\lambda(B_1)$ and $\lambda(\Sigma)$ = $\lambda(\Lambda)$ and the homeomorphisms $h_i$ are conservative, then $h$ can be chosen conservative too.
\end{lemma}

\begin{figure}[ht]
    \centering
    \begin{tikzpicture}[scale=0.8]
        \draw[rounded corners=4mm, fill=blue!30, draw=blue, line width=1pt] (0,0) to[in=190, out=330] (4,0) to[in=240] (4,4) to[in=20, out=175] (0,4) to[in=70, out=300] (0,0);

        \draw[rounded corners=2mm, fill=magenta!30, draw=magenta, line width=1pt] (1,1) to[in=160, out=20] (3,1) to[in=280] (3,3) to[in=330] (1,3) to[in=95, out=270] (1,1);

        \begin{scope}[xshift=7cm]
          \draw[rounded corners=4mm, fill=blue!30, draw=blue, line width=1pt] (0,0) to[in=160, out=30] (4,0) to[in=280] (4,4) to[in=330, out=185] (0,4) to[in=120, out=260] (0,0);

           \draw[rounded corners=2mm, fill=magenta!30, draw=magenta, line width=1pt] (0.5,1) to[in=180, out=30] (2.5,1) to[in=270, out=100] (2.5,3) to[in=20, out=170] (0.5,3) to[in=95, out=270] (0.5,1);            
        \end{scope}

        \draw[->, line width=0.8pt] (3.8,4.8) arc (110:60:4);
        \draw[->, thick] (2.8, 1.4) arc (240:300:5);

        \node[scale=1.3] at (2.8, 1.7) {$A_1$};
        \node[scale=1.3] at (8, 1.7) {$B_1$};
        \node[scale=1.2] at (5.5, 0.4) {$h_1$};
        \node[scale=1.3] at (3.5, 4.5) {$A_2$};
        \node[scale=1.3] at (7.5, 4.3)  {$B_2$};
        \node[scale=1.2] at (5.5, 5.3) {$h_2$};

        \node[scale=1.7] at (0.7, 0.2) {$\Sigma$};
        \node[scale=1.7] at (11, 0.7) {$\Lambda$};

        \node[scale=1.2] at (1.4, 2.5) {\textcolor{magenta}{$\sigma_1$}};
        \node[scale=1.2] at (0.7, 3.9) {\textcolor{blue}{$\sigma_2$}};
        \node[scale=1.2] at (9.1, 2.7) {\textcolor{magenta}{$\tau_1$}};
        \node[scale=1.2] at (10.7, 3.5) {\textcolor{blue}{$\tau_2$}};
    \end{tikzpicture}
    \caption{An example of bicollared embeddings and a conservative local modification.}
    \label{fig-local-modification}
\end{figure}
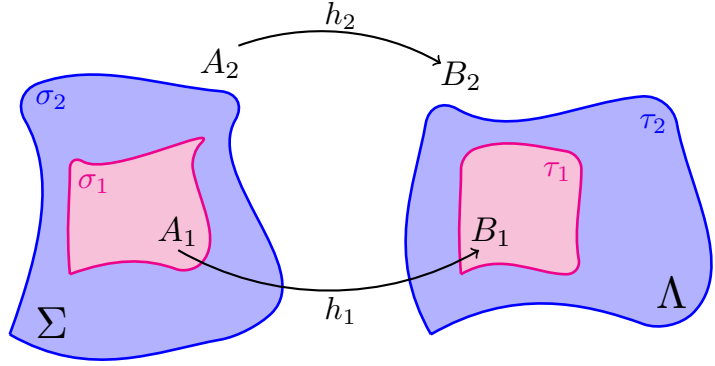

We use this lemma in the following way: let $A_2 = B_2$ and set $h_2$ as the identity in $\mathbb{R}^n$. Consider $h_1: A_1 \to B_1$ such that both $A_1$ and $B_1$ are contained in $A_2$ and $h_1$ is Lebesgue conservative. Hence, $\lambda(A_1)= \lambda(B_2)$, $\lambda(\Sigma) = \lambda(\Lambda)$, and there is $h: \mathbb{R}^n \to \mathbb{R}^n$ conservative such that $h \equiv h_2$ on $A_2$ and $h \equiv h_1$ on $A_1$.

%%%%%%%%%%%%%%%%%%%%%%%%%%%%%%%%%%%%%%%%%%%%%%%%%%%%%%%%%%%%%%%%%%%%%%%%%

\section{Pseudo-horseshoes}\label{section-pseudo-horseshoes}

\subsection{Markovian intersections}
The technique for constructing separated sets is based on Smale's horseshoe, which relies on Markov intersections and partitions.

\begin{definition}\label{def-rectangle}
We cal \textit{rectangle} a subset $R \subset \mathbb{R}^n$ such that 
\[
R = [a_1, b_1]\times ...\times [a_n, b_n],
\]
where $a_i < b_i$ are real numbers. A \textit{face} of $R$ is one of the $(n-1)$-dimensional rectangles constituting the boundary of $R$. We call \textit{horizontal} the faces 
\[
\begin{matrix}
   R^- = [a_1, b_1]\times ...\times [a_{n - 1}, b_{n - 1}]\times\{a_n\} & \text{ and }\\
   R^+ = [a_1, b_1]\times ...\times [a_{n - 1}, b_{n - 1}]\times\{b_n\},
\end{matrix}
\]
and \textit{vertical} the others. We say that a rectangle $R' \subset R$ is a \textit{strict horizontal} (resp. vertical) subrectangle of $R$ if the horizontal (resp. vertical) faces of $R'$ are strictly disjoint from those of $R$ and the vertical (resp. horizontal) faces of $R'$ are included in those of $R$.
\end{definition}

Given $u \in \mathbb{R}^n$, we will denote by $\pi_n(u)$ its last coordinate.

\begin{definition}\label{markovian-intersection}
    Let $\phi$ be a homeomorphism of $U$ open set in $\mathbb{R}^n$, $R_1$ and $R_2$ two rectangles of $U$. We say that $\phi(R_1) \cap R_2$ is a \textit{Markovian intersection} if, for $R_2 = [a_1, b_1]\times ...\times [a_n, b_n]$, there exists a strict horizontal subrectangle $S \subset R_1$ such that either
    \[
    \begin{matrix}
    \phi(S^+) \subset \{u; \pi_n(u) > b_n\} \text{ and } \phi(S^-) \subset \{u; \pi_n(u) < a_n\}; & \text{ or }\\
    \phi(S^-) \subset \{u; \pi_n(u) > b_n\} \text{ and } \phi(S^+) \subset \{u; \pi_n(u) < a_n\}.
    \end{matrix}
    \]
    Also consider that $\phi(S) \subset \{u; \pi_n(u) < a_n\} \cup \text{int}(R_2) \cup \{u; \pi_n(u) > b_n\}$; see Figure \ref{fig-markovian-intersection}.
\end{definition}

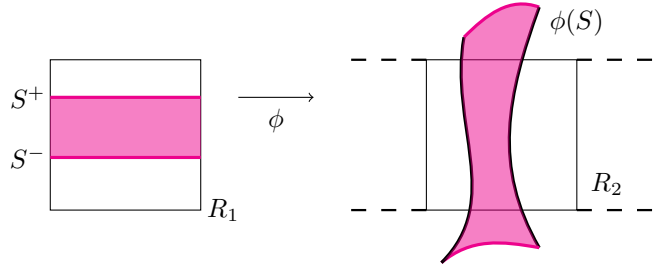
\begin{figure}[ht]
    \centering
    \begin{tikzpicture}
        \draw (0,0) -- (0,2) -- (2,2) -- (2,0) -- cycle;
        \draw[line width=1.2pt, color=magenta] (0,1.5) -- (2,1.5);
        \draw[line width=1.2pt, color=magenta] (0,0.7) -- (2,0.7);
        \fill[magenta, opacity=0.5] (0,0.7) -- (0,1.5) -- (2,1.5) -- (2,0.7);
        \node at (-0.3, 1.5) {$S^+$};
        \node at (-0.3, 0.7) {$S^-$};
        \node at (2.3,0) {$R_1$};

        \draw[->]   (2.5,1.5) -- (3.5,1.5);
        \node[scale=1] at (3, 1.2) {$\phi$};

        \draw (5,0) -- (5,2) -- (7,2) -- (7,0) -- cycle;  
        \node at (7.4,0.35) {$R_2$};
        \draw[dotted, line width=0.3mm, dash pattern=on 2mm off 2mm] (4,0) -- (5,0);
        \draw[dotted, line width=0.3mm, dash pattern=on 2mm off 2mm] (7,0) -- (8,0);
        \draw[dotted, line width=0.3mm, dash pattern=on 2mm off 2mm] (4,2) -- (5,2);
        \draw[dotted, line width=0.3mm, dash pattern=on 2mm off 2mm] (7,2) -- (8,2);

        \draw[line width=1.2pt, color=magenta] (5.5,2.3) to[in=160] (6.5,2.7);
        \draw[line width=1.2pt] (6.5,2.7) to[in=120, out=250] (6.5, -0.5);
        \draw[line width=1.2pt, color=magenta] (6.5, -0.5) to[in=40, out=170] (5.2, -0.7);
        \draw[line width=1.2pt] (5.2, -0.7) to[in=260] (5.5,2.3);
        \fill[magenta, opacity=0.5] (5.5,2.3) to[in=160] (6.5,2.7) to[in=120, out=250] (6.5, -0.5) to[in=40, out=170] (5.2, -0.7) to[in=260] (5.5,2.3);
        \node at (7,2.5) {$\phi(S)$};
    \end{tikzpicture}
    \caption{An example of a Markovian intersection.}
    \label{fig-markovian-intersection}
\end{figure}

\begin{remark}\label{remark-c0-perturbation-markov}
It is well know that the Markovian intersections have nice behaviours under $C^0$ perturbation and iteration. Precisely, if $\phi(R_1) \cap R_2$ is Markovian, then it is still true in a $C^0$ neighborhood of $\phi$; and if for an additional rectangle $R_3$, $\phi(R_2) \cap R_3$ is also Markovian, then the intersection $\phi^2(R_1) \cap R_3$ is Markovian too. To a reference, see \cite{guiheneuf_2012}, Chapter 3.
\end{remark}

Now, define the distance between closed sets $A$ and $B$ in $\mathbb{R}^n$ as $\text{dist}(A, B) = \inf \{\lVert a - b\rVert; a \in A, b \in B\}$.

\begin{definition}\label{epsilon_separated_N_horseshoe}
Let $\delta > 0$ and $\varepsilon_k = \delta/k$, for some integer $k \geq 1$. Given $n \geq 2$, take $U \subset \mathbb{R}^n$ a compact set such that $B_\delta^n \subset \text{Int}\; U$ and consider the positive integer $N_k = (2k)^n$. A homeomorphism $\phi: U \to \mathbb{R}^n$ has an \textit{$\varepsilon_k$-separated $N_k$-pseudo-horseshoe} if for $R_i$ rectangles in $B_\delta^n$, $1 \leq i \leq N_k$, such that $\text{dist}(R_i, R_j) > \varepsilon_k$, it satisfies:
\begin{enumerate}
    \item[i)] $\phi(R_i) \cap R_j \neq \varnothing$ for every $1 \leq i \leq j \leq N_k$;
    \item[ii)] $\phi(R_i) \cap R_j$ is a Markovian intersection;
\end{enumerate}
\end{definition}

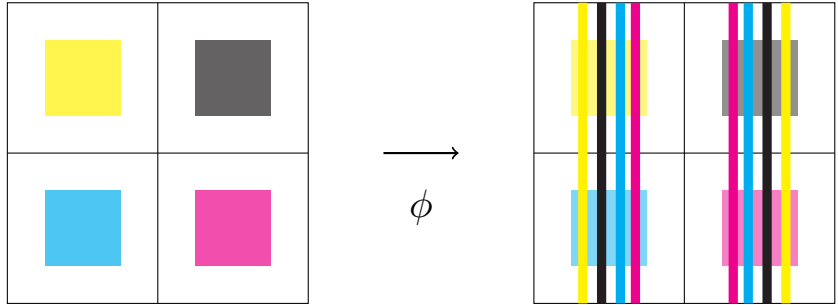
\begin{figure}[ht]
    \centering
\begin{tikzpicture}
  % Definir cores CMYK
  \definecolor{cmykcyan}{cmyk}{1,0,0,0}
  \definecolor{cmykmagenta}{cmyk}{0,1,0,0}
  \definecolor{cmykyellow}{cmyk}{0,0,1,0}
  \definecolor{cmykblack}{cmyk}{0,0,0,1}

  % Ajustar o tamanho do quadrado externo
  \draw (0,0) -- (4,0) -- (4,4) -- (0,4) -- cycle;

  % Desenhar a grade de duas células no quadrado original
  \draw (2,0) -- (2,4);
  \draw (0,2) -- (4,2);

  % Pintar cada quadrado menor no quadrado original
  \fill[cmykcyan, opacity=0.7] (0.5,0.5) rectangle (1.5,1.5);
  \fill[cmykmagenta, opacity=0.7] (2.5,0.5) rectangle (3.5,1.5);
  \fill[cmykyellow, opacity=0.7] (0.5,2.5) rectangle (1.5,3.5);
  \fill[cmykblack, opacity=0.7] (2.5,2.5) rectangle (3.5,3.5);

  % Copiar o quadrado original e posicionar ao lado
  \begin{scope}[xshift=7cm]
    \draw (0,0) -- (4,0) -- (4,4) -- (0,4) -- cycle;
    \draw (2,0) -- (2,4);
    \draw (0,2) -- (4,2);
    \fill[cmykcyan, opacity=0.5] (0.5,0.5) rectangle (1.5,1.5);
    \fill[cmykmagenta, opacity=0.5] (2.5,0.5) rectangle (3.5,1.5);
    \fill[cmykyellow, opacity=0.5] (0.5,2.5) rectangle (1.5,3.5);
    \fill[cmykblack, opacity=0.5] (2.5,2.5) rectangle (3.5,3.5);
    
    % Adicionar uma linha branca que cobre os quadrados ciano e amarelo
    \draw[line width=3.5pt, cmykmagenta] (1.35,4) -- (1.35,0.0);
    \draw[line width=3.5pt, cmykcyan] (1.15,4) -- (1.15,0.0);
    \draw[line width=3.5pt, cmykblack] (0.9,4) -- (0.9,0.0);
    \draw[line width=3.5pt, cmykyellow] (0.65,4) -- (0.65,0.0);

    \draw[line width=3.5pt, cmykmagenta] (2.65,4) -- (2.65,0.0);
    \draw[line width=3.5pt, cmykcyan] (2.85,4) -- (2.85,0.0);
    \draw[line width=3.5pt, cmykblack] (3.1,4) -- (3.1,0.0);
    \draw[line width=3.5pt, cmykyellow] (3.35,4) -- (3.35,0.0);
  \end{scope}

  % Desenhar uma seta entre os quadrados com espaço
  \draw[->, thick] (5,2) -- (6,2);
  
  % Adicionar o símbolo phi (φ) abaixo da seta
  \node[below, scale = 1.5] at (5.5,1.7) {$\phi$};
  \end{tikzpicture}
    \caption{To illustrate an $\varepsilon_1-$separated $N_1-$pseudo-horseshoe.}
      \label{fig-separated-pseudo-horseshoe}
\end{figure}

From Definition \ref{markovian-intersection}, it is easy to see that there is a strict horizontal rectangle $S_{i, j}$ of $R_i$ that contains every $R_i \cap \phi^{-1}(R_j)$. Note that since $\text{dist}(R_i, R_j) > \varepsilon_k$, then for any pair of points $x \in R_{i_1, j_1}$ and $y \in R_{i_2, j_2}$, with $i_1 \neq i_2$ or $j_1 \neq j_2$, we have that $d_2(x, y) \geq \varepsilon_k$. Hence $S(\phi, 2, \varepsilon_k) \geq (N_k)^2$, and inductively, $S(\phi, \ell, \varepsilon_k) \geq (N_k)^\ell$. This estimative is crucial to calculate the metric mean dimension of homeomorphisms.

\subsection{Pseudo-horseshoes on Manifolds}

Recall that we are considering $X$ a smooth manifold of dimension $n \geq 2$. Hence, there is a finite atlas whose charts are bi-Lipschitz. So, from now on, consider $C > 1$ an upper bound of the bi-Lipschitz constants of all charts.

The following definition, an adaptation of Definition 5.3 in \cite{carvalho_rodrigues_varandas_2020}, will extend the notion of horseshoes to a manifold and connect a collection of them along a periodic orbit.

\begin{definition}
Let $X$ be a compact smooth boundaryless manifold of dimension $n \geq 2$ endowed with a distance $d$. Given $f \in \text{Homeo}(X, \mu)$, and constants $\delta > 0$, $k \geq 1$ and $p \in \mathbb{N}$, we say that $f$ has a \textit{chained $(N_k, p)$-pseudo-horseshoe} if there is a pairwise disjoint family of open subsets $\mathcal{U} = \{U_i \subset X; 0 \leq i \leq p - 1\}$ so that
\[
f(U_i) \cap U_{[i + 1]} \neq \varnothing \text{ for all } 0 \leq i \leq p - 1,
\]
where $[i + 1] := (i + 1)\mod p$, and a collection of homeomorphisms $\Psi = \{\psi_i: U_i \subset X \to \mathbb{R}^n; 0 \leq i \leq p - 1\}$ such that $B_\delta^n \subset \psi_i(U_i)$, satisfying, for every $0 \leq i \leq p - 1$, the following itens:
\begin{itemize}
  \item[i)] $(f \circ \psi_i^{-1})\big( B_\delta^n\big) \subset U_{[i + 1]}$;
  \item[ii)] The homeomorphism
  \[
    \phi_i = \psi_{[i + 1]}\circ f\circ\psi_i^{-1}: \psi_i(U_i) \to \mathbb{R}^n
  \]
    has an $\varepsilon_k$-separated $N_k$-pseudo-horseshoe.
\end{itemize}
\end{definition}

Set $\mathcal{R}^i_j = \psi_i^{-1}(R_j)$, for $0 \leq i \leq p - 1$ and $1 \leq j \leq N_k$, and note that $f(\mathcal{R}^i_{j_1}) \cap \mathcal{R}^{[i + 1]}_{j_2}$ is a Markovian intersection on the manifold $X$, for every $1 \leq j_1 \leq j_2 \leq N_k$. Therefore, by Remark \ref{remark-c0-perturbation-markov}, there exists a neighborhood $\mathcal{U} \subset \text{Homeo}(X, \mu)$ of $f$ in the $C^0$-topology such that every $g \in \mathcal{U}$ also has a chained $(N_k, p)$-pseudo-horseshoe.

%%%%%%%%%%%%%%%%%%%%%%%%%%%%%%%%%%%%%%%%%%%%%%%%%%%%%%
\subsection{Separating sets}

\begin{proposition}\label{separating_sets_pseudo_horseshoe}
Given $X$ a smooth compact manifold of dimension $n$ and $f \in \text{Homeo}(X, \mu)$ with a chained $(N_k, p)$-pseudo-horseshoe, then
\[
S(f, m, C\varepsilon_k) \geq \left(\frac{2\delta}{\varepsilon_k}\right)^{n\cdot m}
\]
for all $m \in \mathbb{N}$, where $C > 1$ is an upper bound of the bi-Lipschitz constants of all charts.
\end{proposition}
\begin{proof}
Since $X$ is a smooth manifold, assume that all the charts $\psi_i$ are bi-Lipschitz with Lipschitz constant upper bounded by a uniform constant $C > 1$. That is, given $x, y \in U_i$, for every $0 \leq i \leq p - 1$, $\lVert\psi_i(x) - \psi_i(y)\rVert \leq Cd(x, y)$. Therefore, for a fixed index $i$ and $1 \leq j_1 < j_2 \leq N_k$, $\text{Dist}(\mathcal{R}^i_{j_1}, \mathcal{R}^i_{j_2}) \geq C\varepsilon_k$, where $\text{Dist}(A, B) = \inf\{d(a, b); a \in A, b \in B\}$, $A$ and $B$ closed sets in $X$, and by Definition \ref{epsilon_separated_N_horseshoe}. 

Note that $f(\mathcal{R}^i_{j_1}) \cap \mathcal{R}^{[i + 1]}_{j_2} \neq \varnothing$ for every $0 \leq i \leq p - 1$ and $1 \leq j_1 \leq j_2 \leq N_k$, because each $\psi_{[i + 1]}\circ f\circ\psi_i^{-1}$ has an $\varepsilon_k$-separated $N_k$-pseudo-horseshoe. Therefore, define the following sequence of non-empty nested compact sets:

\[
\begin{matrix}
    j \in \{1, ..., N_k\} & \mapsto & \mathcal{K}^0_j = \mathcal{R}^0_j\\
    j_1, j_2 \in \{1, ..., N_k\} & \mapsto & \mathcal{K}^1_{j_1, j_2} = f^{-1}(f(\mathcal{K}^0_{j_1}) \cap \mathcal{R}^1_{j_2})\\
    \vdots & & \vdots\\
    j_1, ..., j_m \in \{1, ..., N_k\} & \mapsto & \mathcal{K}^{m - 1}_{j_1, ..., j_m} = f^{-(m - 1)}(f^{m - 1}(\mathcal{K}^{m - 2}_{j_1, ..., j_{m - 1}}) \cap \mathcal{R}^{[m - 1]}_{j_m}).
\end{matrix}
\]

Note that $\mathcal{K}^{m - 1}_{j_1, ..., j_m} \cap \mathcal{K}^{m - 1}_{\ell_1, ..., \ell_m} = \varnothing$ if $j_s \neq \ell_s$, for some $1 \leq s \leq m$. Choose $(N_k)^m$ distinct points in $X$ such that $x \in \mathcal{K}^{m - 1}_{j_1, ..., j_m}$ and $y \in \mathcal{K}^{m - 1}_{\ell_1, ..., \ell_m}$, therefore
\[
d_m(x, y) \geq \text{Dist}(\mathcal{R}^{[s - 1]}_{j_s}, \mathcal{R}^{[s - 1]}_{\ell_s}) \geq C\varepsilon_k.
\]
This proves that
\[
S(f, m, C\varepsilon_k) \geq (N_k)^m.
\]
Since $N_k = \left(\frac{2\delta}{\varepsilon}\right)^n$, the proposition is proven.
\end{proof}

%%%%%%%%%%%%%%%%%%%%%%%%%%%%%%%%%%%%%%%%%%%%%%%%%%%%%%%%%%%%%%%%%%%%%%%%%%%%%%%%
\section{Perturbation lemmas}\label{section-perturbation-lemmas}

\begin{definition}\label{def-pseudo-rectangle}
    A set $P \subset \mathbb{R}^n$ is a \textit{pseudo-rectangle} if there is a homeomorphism $\psi: \mathbb{R}^n \to \mathbb{R}^n$ and a rectangle $R$ such that $\psi(R) = P$. And the homeomorphism $\kappa: \mathbb{R}^n \to \mathbb{R}^n$ is \textit{regular for P} if there is an $\varepsilon > 0$ such that for every face $\{R_i^*\}_{i = 1}^{n}$ of the rectangle $R$, there is a $c_i^*$ real number, where $* \in \{-, +\}$, such that 
    \[
    \text{dist}(\kappa\circ\psi(R_i^*), B) \leq \varepsilon, \text{ for some rectangle } B \subset \mathbb{R}^{i - 1}\times\{c_i^*\}\times\mathbb{R}^{n - i}.
    \]
    Finally, we say that $\kappa(P)$ is a \textit{regular pseudo-rectangle}.
\end{definition}

The idea is that the horizontal (resp. vertical) faces of the pseudo-rectangle $\kappa(P)$ are not far from the horizontal (resp. vertical) faces of another rectangle $R'$, just as in Figure \ref{fig-pseudo-rectangle}.

\begin{figure}[ht]
    \centering
    \begin{tikzpicture}
    \draw[line width=1.1pt, color=cyan]     (0,0) -- (0,2);
    \draw[line width=1.1pt, color=magenta]  (0,2) -- (2,2);
    \draw[line width=1.1pt, color=yellow]   (2,2) -- (2,0);
    \draw[line width=1.1pt, color=black]    (2,0) -- (0,0);

    \node[scale=1] at (2.2, 0) {$R$};

    \draw[->]   (2.5,1.5) -- (3.5,1.5);
    \node[scale=1] at (3, 1.2) {$\psi$};

    \draw[line width=1.1pt, color=black]   (4,0.5) to[in=110, out=240] (4,0);
    \draw[line width=1.1pt, color=black]   (4,0) to[in=210] (6,0);
    \draw[line width=1.1pt, color=yellow]   (6,0) to[in=290] (6,2);
    \draw[line width=1.1pt, color=yellow]   (6,2) to[in=350,out=190] (4,2);
    \draw[line width=1.1pt, color=magenta]   (4,2) to[in=80, out=280] (4,1);
    \draw[line width=1.1pt, color=magenta]   (4,1) to[in=200] (4.5,1);
    \draw[line width=1.1pt, color=magenta]   (4.5,1) to[in=190] (5.5,1.7);
    \draw[line width=1.1pt, color=cyan]   (5.5,1.7) to[in=70, out=290] (5.5,0.5);
    \draw[line width=1.1pt, color=cyan]   (5.5,0.5) to[in=50,out=120] (4,0.5);

    \node[scale=1] at (6.3, 0) {$P$};

    \draw[->] (6.7,1.5) -- (7.7,1.5);
    \node[scale=1] at (7.2, 1.2) {$\kappa$};

    \draw[line width=1.1pt, color=cyan]     (8.5,0) to[in=350, out=180] (8.5,1);
    \draw[line width=1.1pt, color=cyan]     (8.5,1) to[in=320, out=10] (8.5,2);
    \draw[line width=1.1pt, color=magenta]  (8.5,2) to[in=110, out=5] (9.5,2);
    \draw[line width=1.1pt, color=magenta] decorate [decoration={snake}]  {(9.5,2) -- (10.5,2)};
    \draw[line width=1.1pt, color=yellow]   (10.5,2) to[in=110, out=280] (10.5,1);
    \draw[line width=1.1pt, color=yellow]   (10.5,1) to[in=30, out=280] (10.5,0);
    \draw[line width=1.1pt, color=black]    (10.5,0) to[in=70,out=120] (9.5,0);
    \draw[line width=1.1pt, color=black] decorate [decoration={snake}]   {(9.5,0) -- (8.5,0)};

    \node[scale=1] at (11,0) {$\kappa(P)$};
    \end{tikzpicture}
    \caption{An example of a homeomorphism $\kappa$ that is regular for the pseudo-rectangle $P$.}
    \label{fig-pseudo-rectangle}
\end{figure}
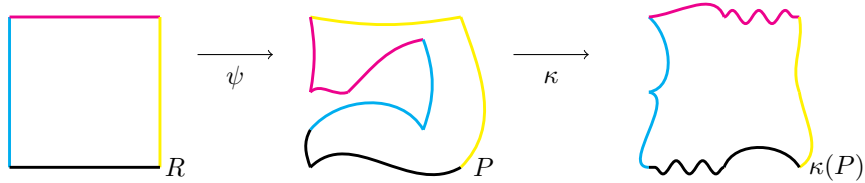

\begin{definition}
A Lebesgue preserving homeomorphism $\Gamma: \mathbb{R}^n \to \mathbb{R}^n$ is a \textit{local conservative rigid-linear modification} if there is $U$ an open set such that:
\begin{enumerate}
    \item[i)] $\Gamma|_U$ is a linear transformation $T$, where $\det T = 1$, or;
    \item[ii)] $\Gamma|_U$ is a translation;
\end{enumerate}
and there is $V$ a neighborhood of $U \cup \Gamma(U)$ such that $\Gamma|_{V^\complement}$ is the identity.
\end{definition}

The existence of such homeomorphism is guaranteed by Lemma \ref{local-modification-lemma}, when both $\partial U$ and $\partial V$ are bicollared embeddings of the sphere. We are also interested in finite composition of such homeomorphisms, as represented in Figure \ref{fig-local-consivative-rigid-linear-modf}.

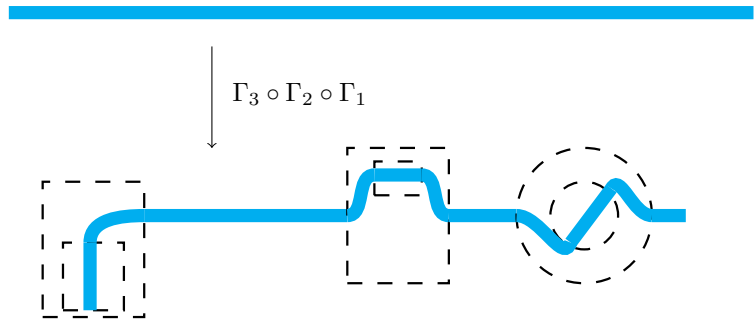
\begin{figure}[ht]
    \centering
    \begin{tikzpicture}[scale=0.9]
        \draw[line width=5pt, color=cyan] (2,0) -- (13,0);

        \draw[->]   (5,-0.5) -- (5,-2);
        \node at (6.3,-1.2) {$\Gamma_3\circ\Gamma_2\circ\Gamma_1$};

        \draw[dotted, line width=0.3mm, dash pattern=on 2mm off 2mm] (2.5,-2.5) -- (2.5,-4.5) -- (4,-4.5) -- (4, -2.5) -- cycle;
        \draw[dotted, line width=0.3mm, dash pattern=on 2mm off 2mm] (2.8,-3.4) -- (2.8,-4.4) -- (3.7,-4.4) -- (3.7, -3.4) -- cycle;
        \draw[line width=5pt, color=cyan] (3.2,-3.4) -- (3.2, -4.4);
        \draw[line width=5pt, color=cyan] (3.2,-3.4) to[in=180, out=90] (4,-3);
        \draw[line width=5pt, color=cyan] (4,-3) -- (7,-3);
        \draw[dotted, line width=0.3mm, dash pattern=on 2mm off 2mm] (7,-2) -- (7,-4) -- (8.5,-4) -- (8.5, -2) -- cycle;
        \draw[dotted, line width=0.3mm, dash pattern=on 2mm off 2mm] (7.4,-2.2) -- (7.4,-2.7) -- (8.1,-2.7) -- (8.1, -2.2) -- cycle;
        \draw[line width=5pt, color=cyan] (7.4,-2.4) -- (8.1,-2.4);
        \draw[line width=5pt, color=cyan] (7,-3) to[in=180, out=0] (7.4,-2.4);
        \draw[line width=5pt, color=cyan] (8.1,-2.4) to[in=180, out=0] (8.5,-3);
        \draw[line width=5pt, color=cyan] (8.5,-3) -- (9.5,-3);
        \draw[dotted, line width=0.3mm, dash pattern=on 2mm off 2mm] (10.5,-3) circle (1);
        \draw[dotted, line width=0.3mm, dash pattern=on 2mm off 2mm] (10.5,-3) circle (0.5);
        \draw[line width=5pt, color=cyan] (11.5,-3) -- (12, -3);
        \draw[line width=5pt, color=cyan] (10.3,-3.4) -- (10.9,-2.6);
        \draw[line width=5pt, color=cyan] (9.5,-3) to[in=250, out=0] (10.3,-3.4);
        \draw[line width=5pt, color=cyan] (10.9,-2.6) to[in=180, out=60] (11.5,-3);
    \end{tikzpicture}
    \caption{An example of a composition of local conservative rigid-linear modifications.}
    \label{fig-local-consivative-rigid-linear-modf}
\end{figure}

\begin{example}
We are interested in Lebesgue measure preserving linear transformations such as rotations, shear and hyperbolic matrices:

\begin{minipage}{0.3\textwidth}
    \centering
    \[R =
    \begin{bmatrix}
    1 & 0 & 0 \\
    0 & 0 & -1 \\
    0 & 1 & 0 \\
    \end{bmatrix},
    \]
  \end{minipage}%
  \begin{minipage}{0.3\textwidth}
    \centering
    \[S =
    \begin{bmatrix}
    1 & 0 & 0 \\
    0 & 1 & 0 \\
    \alpha & \alpha & 1 \\
    \end{bmatrix},
    \]
  \end{minipage}%
  \begin{minipage}{0.3\textwidth}
    \centering
    \[H = 
    \begin{bmatrix}
    3 & 0 & 0 \\
    0 & 2 & 0 \\
    0 & 0 & \frac{1}{6} \\
    \end{bmatrix}.
    \]
  \end{minipage}
\end{example}

\begin{lemma}\label{regular-lebesgue-conservative-homeo}
    For a finite family $\{P_i\}_{i = 1}^M$ of disjoint pseudo-rectangles, there is a $\kappa: \mathbb{R}^n \to \mathbb{R}^n$ Lebesgue conservative homeomorphism that is regular for each $P_i$.
\end{lemma}
\begin{proof}
    By Definition \ref{def-pseudo-rectangle}, there is a homeomorphism $\psi_1$ and a rectangle $R_1$ such that $\psi_1(R_1) = P_1$. Suppose, without loss of generality, that $R_1 = [a, b]^n$, for some $a < b$ real numbers. For $k \in \mathbb{N}$, consider the partition $\mathcal{P}_{a, b} = \{t_j \in \mathbb{R}; t_j = a + j(b - a)/k\}$ of the interval $[a, b]$, where $0 \leq j \leq k$. We can also suppose that $k$ is large enough such that the partition is $\varepsilon$-dense, for some $\varepsilon > 0$.

    Now, apply Lemma \ref{extension-finite-maps-lemma} in $\mathbb{R}^n$ to obtain a Lebesgue conservative homeomorphism $\Phi_1: \mathbb{R}^n \to \mathbb{R}^n$ such that $\Phi_1\circ\psi_1(t_{j_1}, ..., t_{j_n}) = (s_{j_1}, ..., s_{j_n})$, where $s_j$ belongs to the partition $\mathcal{P}_{c, d}$ of the interval $[c, d]$ and $\lambda([c, d]^n) = \lambda(P_1)$. Note that, in particular, $\Phi_1\circ\psi_1(a, a, ..., a) = (c, c, ..., c)$, and $\Phi_1(R_1)$ still a pseudo-rectangle. 

    Use the above step for all the pseudo-rectangles in $\{P_i\}_{i = 1}^M$ to construct a $\Phi: \mathbb{R}^n \to \mathbb{R}^n$ such that the images of all pseudo-rectangles under $\Phi$ are sufficiently spaced apart.

    For the next step, set $R_1^+ = [a, b]^{n - 1}\times\{b\}$ and consider $F := \Phi\circ\psi_1(R_1^+)$. Also consider a finite family of local conservative rigid-linear modifications $\{\Gamma_\ell\}_{\ell = 1}^N$ such that $\Gamma_N\circ...\circ\Gamma_1(F) \subset [c, d]^{n - 1}\times (d, +\infty)$. Moreover, we can choose $\Gamma := \Gamma_N\circ...\circ\Gamma_1$ so that it fix the vertex of the pseudo-rectangle $\Phi(R_1)$.

    This is possible because once $\Gamma_{K_1}\circ...\circ\Gamma_1(F) \subset ([c, d]^{n - 1}\times (-\infty,d])^\complement$, for some $K_1 < N$, we can shear it to get $\Gamma_{K_2}\circ...\circ\Gamma_1(F) \subset \{u \in \mathbb{R}^n; \pi_n(u) \geq d\}$, where $\pi_n$ is the projection on the last coordinate and $K_1 < K_2 < N$. Then, apply locally a conservative hyperbolic matrix of the form $T(u_1, u_2, ..., u_n) = (u_1/\theta, u_2/\theta, ..., \theta^{n-1}u_n)$, where $\theta > 1$ is large enough.

    We argue that $\lambda([c, d]^n) = \lambda(P_1)$ so that the new pseudo-rectangle is not squished into a smaller rectangle and to minimize the number of local modifications.

    Again, repeat the above step for the remain horizontal face and for all the horizontal and vertical faces of all the pseudo-rectangles under $\Phi$. We still call it $\Gamma$.

    To conclude, $\Gamma\circ\Phi$ represents our $\kappa$ because it maps the image of a face of a rectangle, which is a set of the form $[a, b]^{\ell - i}\times{\alpha_\ell}\times[a, b]^{n - \ell}$, where $\alpha_\ell \in \{a, b\}$, to be included in $[c, d]^{\ell - i}\times [-\varepsilon, \varepsilon]\times[c, d]^{n - \ell}$, where $\varepsilon > 0$ can be chosen uniformly for all pseudo-rectangle $P_i$.
\end{proof}

\begin{definition}
    Given two pseudo-rectangles $P_1, P_2$ and $\phi$ an homeomorphism defined in a neighborhood of $P_1$, we say that $\phi(P_1) \cap P_2$ is a \textit{Markovian intersection for $P_1$ and $P_2$} if $\phi\circ\psi_1^{-1}(P_1) \cap \psi_2^{-1}(P_2)$ is a Markovian intersection of rectangles.
\end{definition}

\begin{lemma}\label{existence-chained-pseudo-horseshoe}
Let $F \in \text{Homeo}(X, \mu)$ be such that it has a periodic point of period $p$, then for $\eta > 0$ and every integer $k \geq 1$, there is $f \in \text{Homeo}(X, \mu)$ such that it has a chained $(N_k, p)-$pseudo-horseshoe and $D(F, f) < \eta$.
\end{lemma}
\begin{proof}

Consider $x \in X$ a periodic point for $F$ of period $p$ and $U_i$ a neighborhood of $F^i(x)$, for $0 \leq i \leq p - 1$. Recall that the existence of a periodic point is a consequence of Lemma \ref{extension-finite-maps-lemma} and the Poincaré recurrence theorem.

Choose a bi-Lipschitz chart $\psi_i: U_i \to \mathbb{R}^n$ such that $\psi_i(F^i(x)) = 0$. Suppose that $\psi_i(U_i) = I^n$ and endow it with the measure $\widetilde{\mu}_i = D_i\cdot (\psi_i)_*\mu$, that is, for $A \subset \psi_i(U_i)$ a Borel set, $\widetilde{\mu}_i(A) = D_i\cdot\mu(\psi_i^{-1}(A))$, where $D_i = (\mu(U_i))^{-1}$. Observe that $\widetilde{\mu}_i$ still a good Borel probability measure.

Fix some $i \in \{0, 1, ..., p - 1\}$, choose an integer $k \geq 2$ and consider $B_\delta^n \subset \psi_i(U_i)$, for some $0 < \delta < 1/2$. One can choose sufficiently small $\delta > 0$ such that $F\circ\psi_i^{-1}(B_\delta^n) \subset U_{[i + 1]}$, where $[i + 1] := i + 1 \mod p$. Remember that we defined $\varepsilon_k = \delta/k$ e $N_k = (2k)^n$.

Let $R_j^i$ be rectangles in $B_\delta^n \subset \psi_i(U_i)$ such that $\text{dist}(R_j, R_\ell) > \varepsilon_k$, where $1 \leq j < \ell \leq N_k$. Also consider rectangles $R_j^{[i + 1]}$ with the same property in $B_\delta^n \subset \psi_{[i + 1]}(U_{[i + 1]})$

By the light of Theorem \ref{homeomorphic-measure}, there is a homeomorphism 
\[
\xi_{[i + 1]}: \psi_{[i + 1]}(U_{[i + 1]}) \to \psi_{[i + 1]}(U_{[i + 1]})
\]
such that $\widetilde{\mu}_{[i + 1]} = (\xi_{[i + 1]})_*\lambda$, where $\lambda$ is the Lebesgue measure on $\psi_{[i + 1]}(U_{[i + 1]}) = I^n$, and $\xi_{[i + 1]}$ is the identity on $\partial I^n$.

By the construction in Lemma \ref{regular-lebesgue-conservative-homeo}, there is $\tau_i, \kappa_{[i + 1]}: I^n \to I^n$ Lebesgue conservative homeomorphisms such that $\tau_i$ is regular for each $P^i_j := \xi_{[i + 1]}\circ\psi_{[i + 1]}\circ F\circ\psi_i^{-1}(R_j^i)$, and $\kappa_{[i + 1]}$ is regular for each $P_j^{[i + 1]} := \xi_{[i + 1]}(R_j^{[i + 1]})$.

Now, consider $T: \mathbb{R}^n \to \mathbb{R}^n$ a linear transformation such that 
\[
T(u_1, u_2, ..., u_n) = \left(\theta^{n - 1}u_1, \frac{u_2}{\theta}, ..., \frac{u_n}{\theta}\right)
\]
for some $\theta > 1$. Note that $T$ preserves the Lebesgue measure because $\det(T) = 1$.

Also consider a finite family of local conservative rigid-linear modifications $\Gamma_1, ..., \Gamma_t$ such that $\Gamma = \Gamma_1\circ ...\circ \Gamma_t$ has the following properties:
\begin{enumerate}
    \item[I)] $\Gamma\circ T\circ\tau_i(P^i_j) \cap \kappa_{[i + 1]}(P_\ell^{[i + 1]}) \neq \varnothing$ for all $1 \leq j \leq \ell \leq N_k$;
    \item[II)] $\Gamma\circ T\circ\tau_i(P^i_j) \cap \kappa_{[i + 1]}(P_\ell^{[i + 1]})$ is a Markovian intersection for $P^i_j$ and $P_\ell^{[i + 1]}$.
\end{enumerate}

Apply Lemma \ref{local-modification-lemma} to obtain a Lebesgue conservative homeomorphism $\mathcal{L}: \mathbb{R}^n \to \mathbb{R}^n$ such that:

\begin{enumerate}
    \item[a)] In a neighborhood $V$ of $\left(\bigcup\limits_{j = 1}^{N_k} \tau_i(P^i_j)\right)$, where 
    \[
    V \subset \tau_i\circ\xi_{[i+1]}\circ\psi_{[i+1]}(F(U_i) \cap U_{[i + 1]})
    \]
    and $\partial V$ is a bicollar embedding of $\mathbb{S}^{n - 1}$, we have that $\mathcal{L}|_V = \Gamma\circ T|_V$;
    \item[b)] Outside $W$, open set such that 
    \[
    (V \cup \Gamma\circ T(V)) \subset W \subset \tau_i\circ\xi_{[i+1]}\circ\psi_{[i+1]}(F(U_i) \cap U_{[i + 1]})
    \]
    and $\partial W$ is also a bicollar embedding of $\mathbb{S}^{n - 1}$, $\mathcal{L}$ is the identity.
\end{enumerate}

See Figure \ref{fig-chained-pseudo-horseshoe} to have a illustrated overview of this proof. Therefore, $\xi_{[i+1]}^{-1}\circ\kappa_{[i+1]}^{-1}\circ\mathcal{L}\circ\tau_i\circ\xi_{[i + 1]}\circ\psi_{[i+1]}\circ F\circ\psi_i^{-1}: I^n \to I^n$ has an $\varepsilon_k$-separated $N_k$-horseshoe and $f: X \to X$ defined as 
\[
f(x) =
\begin{cases}
    \psi_{[i+1]}^{-1}\circ\xi_{[i+1]}^{-1}\circ\kappa_{[i+1]}^{-1}\circ\mathcal{L}\circ\tau_i\circ\xi_{[i + 1]}\circ\psi_{[i+1]}\circ F(x), & \text{if}\;\; x \in U_i;\\
    F(x), & \text{otherwise};
\end{cases}
\]
for all $0 \leq i \leq p - 1$, is such that $f \in \text{Homeo}(X, \mu)$ and it also has a chained $(N_k, p)$-pseudo-horseshoe.

Finally, in order to obtain $D(F, f) < \eta$ one can shrink each $U_i$ so that 
\[
\max\limits_{0 \leq i \leq p - 1} \text{diam}(U_i) < \eta,
\] 

where $\text{diam}(A) = \sup_{x, y \in A} d(x, y)$ and $d$ is the Riemannian metric.
\end{proof}

\begin{figure}[ht]
    \centering
    \begin{tikzpicture}[scale=0.8]
        \draw[line width=.7pt] (0,0) -- (0,3) -- (3,3) -- (3,0) -- cycle;
        \draw[line width=.7pt, purple] (0.5,0.5) -- (0.5,2.5) -- (2.5,2.5) -- (2.5,0.5) -- cycle;
        \fill[magenta, opacity=0.7] (0.7,0.7) -- (0.7,1.2) -- (1.2,1.2) -- (1.2,0.7) -- cycle;
        \fill[magenta, opacity=0.7] (2.3,0.7) -- (2.3,1.2) -- (1.8,1.2) -- (1.8,0.7) -- cycle;
        \fill[magenta, opacity=0.7] (0.7,2.3) -- (0.7,1.8) -- (1.2,1.8) -- (1.2,2.3) -- cycle;
        \fill[magenta, opacity=0.7] (1.8,1.8) -- (1.8,2.3) -- (2.3,2.3) -- (2.3,1.8) -- cycle;

        \draw[->, thick] (3.8,2) -- (5.2,2);
        \node[scale=0.8] at (4.5,1.6) {$\psi_{[i+1]}\circ F\circ \psi_i^{-1}$};

        \node at (3.4,-0.3) {$\psi_i(U_i)$};

        \begin{scope}[xshift=6cm]
        \draw[line width=.7pt] (0,0) -- (0,3) -- (3,3) -- (3,0) -- cycle;
        \fill[lightgray, opacity=0.2] (0,0) -- (0,3) -- (3,3) to[in=0,out=270] cycle;
        \draw[line width=.7pt, purple] (0.3,0.3) to[in=260,out=90] (0.4,2.7) to[in=180, out=340] (2.5,2.9) to[in=80, out=280] (2.2,1.1) to[in=10,out=190] cycle;
        \fill[magenta, opacity=0.7] (0.4,0.4) to[in=270] (0.6,1.2) to[in=150] (1.2,1.2) to[in=70,out=280] (1,0.7) -- cycle;
        \fill[magenta, opacity=0.7] (2.1,1.2) to[in=110] (2.1,1.4) to[in=70,out=120] (1.4,1.5) to[in=70, out=280] (1.4,1) to[in=170] cycle;
        \fill[magenta, opacity=0.7] (0.6,2.2) to[in=110,out=260] (0.6,1.7) to[in=180] (1.1,2) to[in=240] (1.1,2.5) to[in=80] cycle;
        \fill[magenta, opacity=0.7] (1.8,2) to[in=280] (1.8,2.6) to[in=160] (2.3,2.6) to[in=110,out=270] (2.3,2.3) to[in=0] cycle;
        \end{scope}

        \begin{scope}[xshift=10.5cm]
        \draw[line width=.7pt] (0,0) -- (0,3) -- (3,3) -- (3,0) -- cycle;
        \draw[line width=.7pt, blue] (0.5,0.5) -- (0.5,2.5) -- (2.5,2.5) -- (2.5,0.5) -- cycle;
        \fill[cyan, opacity=0.7] (0.7,0.7) -- (0.7,1.2) -- (1.2,1.2) -- (1.2,0.7) -- cycle;
        \fill[cyan, opacity=0.7] (2.3,0.7) -- (2.3,1.2) -- (1.8,1.2) -- (1.8,0.7) -- cycle;
        \fill[cyan, opacity=0.7] (0.7,2.3) -- (0.7,1.8) -- (1.2,1.8) -- (1.2,2.3) -- cycle;
        \fill[cyan, opacity=0.7] (1.8,1.8) -- (1.8,2.3) -- (2.3,2.3) -- (2.3,1.8) -- cycle;

        \node at (3.4,-0.3) {$\psi_{[i+1]}(U_{[i+1]})$};
        \end{scope}

        \draw[->] (7.5,-0.5) -- (7.5,-1);
        \node at (8.1, -0.8) {$\xi_{[i + 1]}$};

        \begin{scope}[shift={(6,-4.5)}]
        \draw[line width=.7pt] (0,0) -- (0,3) -- (3,3) -- (3,0) -- cycle;
        \fill[lightgray, opacity=0.2] (0,0) -- (0,3) -- (3,3) to[in=90,out=260] (2.5,0.5) to[in=30,out=180] cycle;
        \fill[magenta, opacity=0.7] (0.2,0.2) to[in=260, out=110] (0.3,1.2) to[in=180] (0.9,1.5) to[in=140] (1.3,0.7) to[in=80, out=180] cycle;
        \fill[magenta, opacity=0.7] (2.3,0.7) to[in=110] (2.3,1.7) to[in=10, out=180] (1.4,1.2) to[in=90, out=330] (1.8,0.7) to[in=170] cycle;
        \fill[magenta, opacity=0.7] (0.2,2.7) to[in=120,out=280] (0.4,1.8) to[in=110] (1.3,1.8) to[in=290] (1.1,2.5) to[in=330] cycle;
        \fill[magenta, opacity=0.7] (1.8,2) to[in=290] (1.3,2.8) to[in=190] (2.3,2.8) to[in=90] (2.6,2.3) to[in=340, out=180] cycle;  
        \end{scope}

        \draw[->] (12,-0.5) -- (12,-1);
        \node at (12.6, -0.85) {$\xi_{[i + 1]}$};

        \begin{scope}[shift={(10.5,-4.5)}]
        \draw[line width=.7pt] (0,0) -- (0,3) -- (3,3) -- (3,0) -- cycle;
        \fill[cyan, opacity=0.7] (0.4,0.3) to[in=270] (0.6,1.2) to[in=160] (1.2,1.2) to[in=120, out=240] (1.3,0.4) to[in=50, out=140] cycle;
        \fill[cyan, opacity=0.7] (2.9,0.2) to[in=330, out=90] (2.6,1.2) to[in=300] (1.5,1.4) to[in=120, out=270] (1.8,0.5) to[in=100] cycle;
        \fill[cyan, opacity=0.7] (0.4,2.7) to[in=150, out=270] (0.5,1.4) to[in=180] (1.2,1.8) to[in=280] (1.2,2.3) to[in=270] cycle;
        \fill[cyan, opacity=0.7] (1.8,1.8) to[in=270] (1.6,2.7) to[in=170] (2.5,2.7) to[in=160] (2.7,1.8) to[in=330, out=190] cycle;   
        \end{scope}

        \draw[->] (7.5,-5) -- (7.5,-5.5);
        \node at (7.9, -5.3) {$\tau_i$};

        \begin{scope}[shift={(6,-9)}]
        \draw[line width=.7pt] (0,0) -- (0,3) -- (3,3) -- (3,0) -- cycle;
        \fill[lightgray, opacity=0.2] (0,0) -- (0,3) -- (3,3) to[in=40,out=260] (2.5,0.5) to[in=40,out=180] cycle;
        \draw[line width=1pt, orange] (0.3,0.3) to[in=260, out=100] (0.3,2.7) to[in=170, out=20] (2.7,2.8) to[in=60, out=280] (2.5,0.7) to[in=40, out=190] cycle;
        \fill[magenta, opacity=0.7] (0.4,0.8) to[in=240] (0.4,1.3) to[in=300] (0.4,1.6) to[in=160] (0.9,1.6) to[in=170] (1.3,1.6) to[in=60] (1.3,1.3) to[in=120, out=240] (1.3,0.8) to[in=330, out=200] cycle;
        \fill[magenta, opacity=0.7] (2.4,0.9) to[in=240, out=110] (2.4,1.7) to[in=40, out=110] (1.7,1.7) to[in=80, out=280] (1.7,0.9) to[in=160] cycle;
        \fill[magenta, opacity=0.7] (0.5,2.6) to[in=120,out=260] (0.5,2) to[in=200] (1.2,2) to[in=290] (1.2,2.6) to[in=330, out=90] cycle;
        \fill[magenta, opacity=0.7] (1.5,2.1) to[in=280] (1.5,2.6) to[in=190] (2.5,2.6) to[in=40] (2.5,2.1) to[in=30] cycle;
        \end{scope}

        \draw[->] (12,-5) -- (12,-5.5);
        \node at (12.7, -5.4) {$\kappa_{[i + 1]}$};

        \draw[->] (9.5,-7) -- (10, -7);
        \node at (9.7,-7.3) {$\Gamma\circ T$};

        \begin{scope}[shift={(10.5,-9)}]
        \draw[line width=.7pt] (0,0) -- (0,3) -- (3,3) -- (3,0) -- cycle;
        \draw[line width=1pt, orange] (0.3,0.3) to[in=260, out=100] (0.3,2.7) to[in=170, out=20] (2.7,2.8) to[in=60, out=280] (2.5,0.7) to[in=40, out=190] cycle;
        \fill[lightgray, opacity=0.2] (0,0) -- (0,3) -- (3,3) to[in=40,out=260] (2.5,0.5) to[in=40,out=180] cycle;
        \fill[cyan, opacity=0.7] (0.4,0.9) to[in=260] (0.4,1.4) to[in=190] (1.2,1.4) to[in=60] (1.2,0.9) to[in=330, out=190] cycle;
        \fill[cyan, opacity=0.7] (2.5,1) to[in=250] (2.5,1.4) to[in=330] (1.7,1.4) to[in=100, out=270] (1.7,1) to[in=200] cycle;
        \fill[cyan, opacity=0.7] (0.45,2.3) to[in=120, out=260] (0.45,1.6) to[in=160] (1.3,1.6) to[in=290] (1.3,2.3) to[in=290, out=200] cycle;
        \fill[cyan, opacity=0.7] (1.6,1.5) to[in=280] (1.6,2.3) to[in=190] (2.3,2.3) to[in=50, out=330] (2.3,1.5) to[in=20] cycle; 
        \draw[line width=1pt, magenta] (0.5, 0.8) -- (0.6, 2.6) -- (2.2, 2.75) -- (2.2,0.8);
        \draw[line width=1pt, magenta] (0.6, 0.8) -- (0.7, 2.5) -- (2.1, 2.65) -- (2.1,0.8);
        \draw[line width=1pt, magenta] (0.7, 0.8) -- (0.8, 2.4) -- (2, 2.55) -- (2,0.8);
        \draw[line width=1pt, magenta] (0.8, 0.8) -- (0.9, 2.3) -- (1.9, 2.45) -- (1.9,0.8);
        \end{scope}
    \end{tikzpicture}
    \caption{An illustration of the pseudo-horseshoe construction. In the right column, both squares represents the same set. The gray area is the intersection $F(U_i) \cap U_{[i + 1]}$ along the charts and the orange curve represents the set $W$ in the proof.}
    \label{fig-chained-pseudo-horseshoe}
\end{figure}
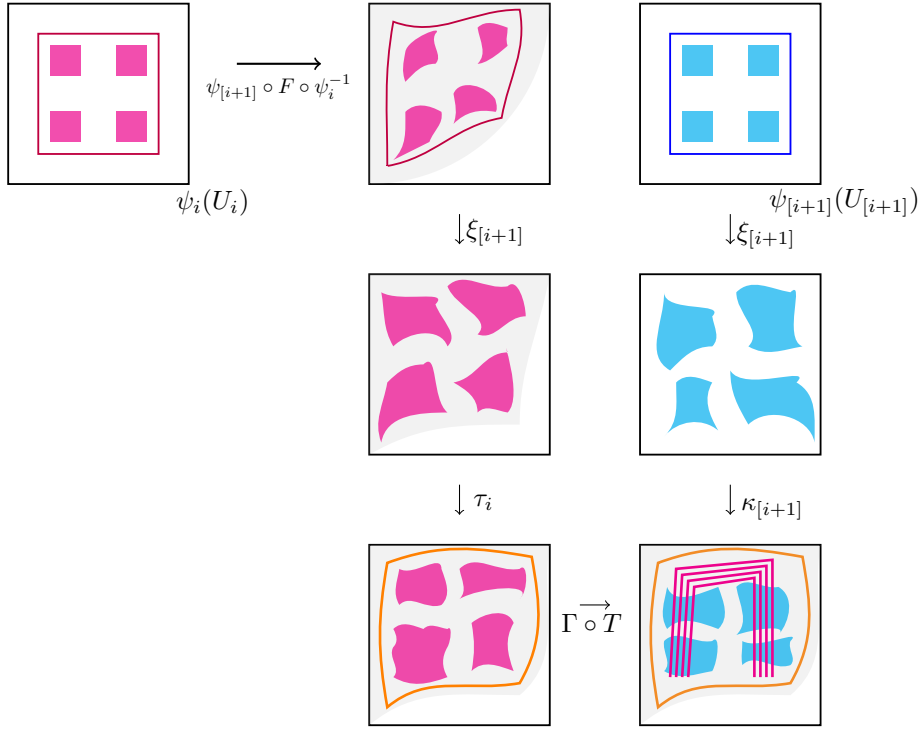

%%%%%%%%%%%%%%%%%%%%%%%%%%%%%%%%%%%%%%%%%%%%%%%%%%%%%%

\section{Proof of the Main Theorems} \label{section-proof-of-the-main-theorems}

\subsection{Proof of Theorem A}

Denote by $\mathcal{H}$ the set $\text{Homeo}(X, \mu)$. Given $F \in \mathcal{H}$, an integer $k \geq 1$, and $\eta > 0$, there is an open set $\mathcal{U}_F^k$ such that every $g \in \mathcal{U}_F^k$ has a chained $(N_k, p)$-pseudo-horseshoe and $D(F, g) < \eta$. This is a consequence of Remark \ref{remark-c0-perturbation-markov} and Lemma \ref{existence-chained-pseudo-horseshoe}.

From Proposition \ref{separating_sets_pseudo_horseshoe}, remember that $C > 1$ is the bi-Lipschitz upper bound of all charts, 
\[
S(g, \ell, C\varepsilon_k) \geq (N_k)^\ell = \left(\frac{2\delta}{\varepsilon_k}\right)^{n\cdot \ell}.
\]

Then
\[
\text{Sep}(g, \varepsilon_k) \geq \text{Sep}(g, C\varepsilon_k) = \limsup\limits_{\ell \to \infty} \frac{\log S(g, \ell, C\varepsilon_k)}{\ell} \geq n (\log 2\delta - \log \varepsilon_k).
\]

Consider the set $\mathcal{U}^k = \bigcup\limits_{F \in \mathcal{H}} \mathcal{U}_F^k$, that is open and dense in $\mathcal{H}$. Therefore, $\mathfrak{R} = \bigcap\limits_{k \in \mathbb{N}} \mathcal{U}^k$ is a residual set in $\mathcal{H}$. If $f \in \mathfrak{R}$, then 
\[
\overline{\text{mdim}}_M(X, f, d) = \limsup\limits_{\varepsilon \to 0} \frac{\text{Sep}(f, \varepsilon)}{-\log\varepsilon} \geq \lim\limits_{k \to \infty} \frac{\text{Sep}(f, \varepsilon_k)}{-\log\varepsilon_k} \geq n - \lim\limits_{k \to \infty}\frac{n\log 2\delta}{\log \varepsilon_k}. 
\]
Remember that $\delta > 0$ depends only on the finite atlas, so the above limit is well defined. Thus, $\overline{\text{mdim}}_M(X, f, d) \geq n$. By Remark 4 in \cite{velozo_velozo_2017}, $n$ is the upper bound for the metric mean dimension. Therefore, for all $f \in \mathfrak{R}$, $\overline{\text{mdim}}_M(X, f, d) = n$.

%%%%%%%%%%%%%%%%%%%%%%%%%%%%%%%%%%%%%%%%%%%%%%%%%%%%%%%%%

\subsection{Proof of Theorem B}

Choose $f \in \mathfrak{R}$, as in the proof of Theorem \ref{conservative-homeo-full-metric-mean-dimension}. By Proposition \ref{separating_sets_pseudo_horseshoe}, there is $\mathcal{S}_k$, $(m, \varepsilon_k)$-separated set such that $\#\mathcal{S}_k = (N_k)^m$. For each $x_i, x_j \in \mathcal{S}_k$, $1 \leq i < j \leq (N_k)^m$, note that $B_{(m, \varepsilon_k/2)}(x_i) \cap B_{(m, \varepsilon_k/2)}(x_j) = \varnothing$. Otherwise, $d_m(x_i, x_j) < \varepsilon_k$.

Set $\delta_k = \min_{1 \leq i \leq (N_k)^m} \mu(B_{(m, \varepsilon_k/2)}(x_i))$ and consider $\mathcal{V}$ a subset of $X$ with $\mu(\mathcal{V}) > 1 - \delta_k$, and $N(\mathcal{V}, f, m, \varepsilon_k/2) = N_\mu(f, m, \varepsilon_k/2, \delta_k)$, where $N(\mathcal{V}, f, m, \varepsilon_k/2)$ is defined to be the minimum number of $(m, \varepsilon_k/2)$-dynamical balls needed to cover $\mathcal{V}$. As in the proof of Lemma 1 in \cite{velozo_velozo_2017}, we can assume that $\mathcal{V}$ is open in $X$. Note that when $k \to \infty$, then $\delta_k \to 0$.

\vspace{5mm}
\textit{Claim.} For all $i \in \{1, ..., (N_k)^m\}$, $\mathcal{V} \cap B_{(m, \varepsilon_k/2)}(x_i) \neq \varnothing$.

Indeed, if there is $j \in \{1, ..., (N_k)^m\}$ such that $\mathcal{V} \cap B_{(m, \varepsilon_k/2)}(x_j) = \varnothing$, then $\mathcal{V} \subset (B_{(m, \varepsilon_k/2)}(x_j))^\complement$. But this implies that $\mu(\mathcal{V}) < 1 - \delta_k$. Contradiction.
\vspace{5mm}

Therefore, $N(\mathcal{V}, f, m, \varepsilon_k/2) = N_\mu(f, m, \varepsilon_k/2, \delta_k) \geq (N_k)^m$, and
\[
h_\mu(\varepsilon_k/2, f, \delta_k) = \limsup\limits_{m \to \infty} \frac{\log N_\mu(f, m, \varepsilon_k/2, \delta_k)}{m} \geq \log N_k.
\]
A similar computation as in the proof of Theorem \ref{conservative-homeo-full-metric-mean-dimension} results in the following inequality
\[
\overline{\text{mdim}}_M(X, f, d)  \geq \lim\limits_{k \to \infty} \frac{h_\mu(\varepsilon_k/2, f, \delta_k)}{-\log\varepsilon_k/2} \geq n.
\]
Finally, $\mu$ attains the supremum in the variational principle stated in Theorem \ref{variational-principle-velozo}.

%%%%%%%%%%%%%%%%%%%%%%%%%%%%%%%%%%%%%%%%%%%%%%%%%%%%%%%%%

\section{Remarks}

As a conclusion, we formulate a few remarks:
\begin{enumerate}
    \item[I)] A very similar proof can be applied to the dissipative case. Therefore, we also provide an alternative proof for Theorem A in \cite{carvalho_rodrigues_varandas_2020}.
    \item[II)] The authors believe that the technique presented in this work helps to prove the density of the level sets of the metric mean dimension in the conservative case, just as it has been done in the dissipative case in \cite{acevedo_romana_arias_2023}.
\end{enumerate}

\printbibliography

@book{alpern_prasad_04,
    author = {S. Alpern and V. S. Prasad},
    title = {Typical dynamics of volume preserving homeomorphisms},
    publisher = {Cambridge University Press},
    year = {2004}
}

@article{carvalho_rodrigues_varandas_2020,
    author = {M. Carvalho and F. B. Rodrigues and P. Varandas},
    title = {Generic homeomorphisms have full metric mean dimension},
    journal = {Ergodic Theory and Dynamical Systems},
    year = {2020},
    volume = {42}, number = {1}, pages = {40-64}
}

@article{guiheneuf_lefeuvre_2018,
    author = {P. Guihéneuf and T. Lefeuvre},
    title = {On the genericity of the shadowing property for conservative homeomorphisms},
    journal = {Proceedings of the American Mathematical Society},
    year = {2018},
    volume = {146}, pages = {4225-4237}
}

@article{oxtoby_ulam_1941,
    author = {J. Oxtoby and S. Ulam},
    title = {Measure-Preserving Homeomorphisms and Metrical Transitivity},
    journal = {Annals of Mathematics},
    year = {1941},
    volume = {42}, number = {4}, pages = {874-920}
}

@article{yano_1980,
    author = {K. Yano},
    title = {A remark on the topological entropy of homeomorphisms},
    journal = {Inventiones mathematicae},
    year = {1980},
    volume = {59}, pages = {215-220}
}

@article{gromov_1999,
    author = {M. Gromov},
    title = {Topological Invariants of Dynamical Systems and Spaces of Holomorphic Maps: I},
    journal = {Mathematical Physics, Analysis and Geometry},
    year = {1999},
    volume = {2}, pages = {323-415}
}

@article{lindenstrauss_weiss_2000,
    author = {E. Lindenstrauss and B. Weiss},
    title = {Mean topological dimension},
    journal = {Israel Journal of Mathematics},
    year = {2000},
    volume = {115}, pages = {1-24}
}

@article{katok_stepin_1970,
    author = {A. Katok and A. Stepin},
    title = {Metric properties of measure preserving homeomorphisms},
    journal = {Russian Mathematical Surveys},
    year = {1970},
    volume = {25}, number = {2}, pages = {191}
}

@article{guiheneuf_2012,
    author = {P. Guihéneuf},
    title = {Propriétés dynamiques génériques des homéomorphismes conservatifs},
    journal = {Ensaios Matemáticos},
    year = {2012},
    volume = {22}, pages = {115}
}

@article{lindenstrauss_tsukamoto_2018,
    author = {E. Lindenstrauss and M. Tsukamoto},
    title = {From Rate Distortion Theory to Metric Mean Dimension: Variational Principle},
    journal = {IEEE Transactions on Information Theory},
    year = {2018},
    volume = {64}, number = {5}, pages = {3590-3609}
}

@article{katok_1980,
    author = {A. Katok},
    title = {Lyapunov exponents, entropy and periodic orbits for diffeomorphisms},
    journal = {Publications Mathématiques de l'Institut des Hautes Études Scientifiques},
    year = {1980},
    volume = {51}, pages = {137-173}
}

@misc{velozo_velozo_2017,
    author = {A. Velozo and R. Velozo},
    title = {Rate distortion theory, metric mean dimension and measure theoretic entropy},
    year={2017},
    eprint={1707.05762},
    archivePrefix={arXiv},
}

@misc{acevedo_romana_arias_2023,
    author = {J. Acevedo and S. Romaña and R. Arias},
    title = {Density of the level sets of the metric mean dimension for homeomorphisms},
    year={2023},
    eprint={2207.11873},
    archivePrefix={arXiv},
}

@incollection{brown_1961,
    author = {M. Brown},
    title = {A mapping theorem for untriangulated manifolds},
    booktitle = {Topology of 3-manifolds and related topics ({P}roc. {T}he {U}niv. of {G}eorgia {I}nstitute, 1961)},
    publisher = {Prentice-Hall, Inc., Englewood Cliffs, NJ},
    year = {1961},
    pages = {92-94}
}

\end{document}